\documentclass[11pt,leqno]{amsart}
\usepackage{graphicx} 
\usepackage{amsmath,amsfonts,amsthm,mathrsfs,amssymb,amscd,comment,enumitem,amsxtra,url,tikz-cd}
\input{xypic}
\xyoption{all}

\usepackage{xcolor} 
\colorlet{mdtRed}{red!50!black}
\definecolor{dblue}{rgb}{0,0,.6}
\usepackage[colorlinks]{hyperref}
\hypersetup{linkcolor=blue,citecolor=dblue,filecolor=dullmagenta,urlcolor=mdtRed}

\newcommand{\mb}{\mathbb}
\newcommand{\mc}{\mathcal}

\newcommand{\Fe}{\mb F_e}
\newcommand{\Fer}{\mb F_{e,r}}

\theoremstyle{plain}
\numberwithin{equation}{section}
\newtheorem{theorem}{Theorem}[section]
\newtheorem*{theorem*}{Theorem}
\newtheorem{proposition}[theorem]{Proposition}
\newtheorem{lemma}[theorem]{Lemma}

\newtheorem{conjecture}[theorem]{Conjecture}

\theoremstyle{definition}
\newtheorem{definition}[theorem]{Definition}

\newtheorem{problem}[theorem]{Problem}

\newtheorem{remark}[theorem]{Remark}

\newtheorem{algorithm}[theorem]{Algorithm}

\title{Linear systems on blow-ups of Hirzebruch surfaces}
\author[C. J. Jacob]{Cyril J. Jacob} 
\thanks{Corresponding author.}
	\address{Department of Mathematics, Indian Institute of Technology Bombay, Powai, Mumbai 
		400076, Maharashtra, India}
	\email{cyriljjacob@iitb.ac.in} 
    
\author[R. Sebastian]{Ronnie Sebastian} 
	\address{Department of Mathematics, Indian Institute of Technology Bombay, Powai, Mumbai 
		400076, Maharashtra, India}
	\email{ronnie@iitb.ac.in} 

\begin{document}
\date{}

	\begin{abstract}
	    Motivated by various equivalent versions of the SHGH conjecture for $\mathbb P^2$
        blown up at very general points, we propose a similar conjecture for 
        Hirzebruch surfaces. We prove that this conjecture is true for the 
        Hirzebruch surface $\Fe$ blown up at $r\leqslant e+5$ very general points. 
	\end{abstract}

	\maketitle
    \section{Introduction}

    Throughout this article, we work over the field of complex numbers. 
    The interpolation problem of plane curves of a given degree passing 
    through given points with at least given multiplicities 
    is a classical problem in algebraic geometry. 
    A conjectural answer to this problem, assuming the points are 
    very general, is given by the SHGH (Segre-Harbourne-Gimigliano-Hirschowitz) 
    conjecture (see \cite{Seg62,Har86,Gim87,Hir89}). 
    
    One may pose a similar interpolation problem on arbitrary surfaces, 
    with the notion of degree replaced by an appropriate numerical invariant. 
    In this article, we focus on Hirzebruch surfaces.
    These surfaces are of particular interest because the minimal rational surfaces other than 
    the projective plane are precisely the Hirzebruch surfaces. 
    For each $e\in \mb Z_{\geqslant 0}$,  
    the surface $$\mb F_e=\mb P(\mc O_{\mb P^1}\oplus \mc O_{\mb P^1}(-e))$$ 
    is a Hirzebruch surface, and every Hirzebruch surface arises in this way.
	Let $C_e$ denote   the divisor class corresponding to $\mc O_{\Fe}(1)$ and 
    let $f$ denote the divisor class corresponding to the fiber of the map $\Fe \to \mb P^1$. 
    Then $\text{Pic }\Fe =\mb Z C_e \oplus \mb Z f.$
    The interpolation problem on Hirzebruch surfaces is the following: 
    \begin{problem}\label{prblm}
        Let $a,b, m_1,m_2,\dots,m_r$ be non-negative integers. Given $r$ 
        points $p_1,p_2,\dots,p_r$ on $\Fe$, does there exist a curve in 
        the linear system $|aC_e+bf|$ passing through $p_i$ with multiplicity 
        at least $m_i$ for all $1\leqslant i \leqslant r$?
    \end{problem}
   More concretely, we are interested in determining the dimension of the 
   linear system of curves in $|aC_e+bf|$. 
    
    Consider $r$ very general points, say $p_1,p_2,\dots,p_r$, on $\Fe$.
	 Let $\pi: \Fer  \to \Fe$ denote the blow-up of $\Fe $ at these points, 
     and let  $E_1,E_2,\dots, E_r$ denote the exceptional divisors. 
     We will slightly abuse notation 
    and denote by $C_e$ and $f$ their pullbacks to $\Fer$ along $\pi$.
    The Picard group of $\Fer$ is 
     $$\text{Pic }\Fer=\mb Z C_e\oplus \mb Z f \oplus \mb Z E_1 \oplus \mb Z E_2 \oplus \dots \oplus \mb Z E_r.$$
     Note that the linear system of curves in $|aC_e+bf|$ 
    passing through $p_i$ with multiplicity at least $m_i$, for all 
    $1\leqslant i \leqslant r$, 
    is in bijective correspondence with the complete linear system 
    $|aC_e+bf-m_1E_1-m_2E_2-\dots-m_rE_r|$ on $\Fer$. 
    Hence, we are interested in finding the dimension of the 
    linear system $|aC_e+bf-m_1E_1-m_2E_2-\dots-m_rE_r|$. Hence, 
    throughout the article, we are interested in divisors in the following form:
\begin{equation}\label{cond}
    \text{$D= aC_e+bf-m_1E_1-\dots -m_rE_r$ on $\Fer$ where 
        $a,b,m_1,\dots ,m_r\in \mathbb Z_{\geqslant 0}$.}
\end{equation}
    \begin{definition}\label{defn}
     The \textit{virtual dimension} of $D$ is defined as
    $$v(D)=\frac{D^2-K_{\Fer}\cdot D}{2}\,.$$
     The \textit{expected dimension} of $D$ is defined as
    $e(D)=\max\left\{ v(D),-1\right\}\,.$
    \end{definition}
    For divisors as in \ref{cond}, we can see that $h^2(\Fer,\mathcal{O}_{\Fer}(D))=0$.
    From the Riemann-Roch theorem for surfaces, we  see that 
    $$\dim |D|=\frac{D^2-K_{\Fer}\cdot D}{2}+h^1(\Fer,\mathcal{O}_{\Fer}(D))=
        v(D)+h^1(\Fer,\mathcal{O}_{\Fer}(D))\geqslant v(D)\,.$$
    Since ${\rm dim}|D|\geqslant -1$, it follows that 
    ${\rm dim}|D|\geqslant e(D)\geqslant v(D)$.
    \begin{definition}\label{non-spcl defn}
    Let $D$ be a divisor as in \eqref{cond}. 
    We say $D$ is \textit{non-special} if $\dim |D|=e(D)$ and \textit{special} otherwise $(\dim |D|>e(D))$.
    \end{definition}
    Thus, it follows that for $D$ as in \eqref{cond}, if 
    $h^1(\Fer,\mathcal{O}_{\Fer}(D))=0$,
    then $D$ is non-special. Further, if $D$ is effective and non-special then 
    $h^1(\Fer,\mathcal{O}_{\Fer}(D))=0$. Thus, for an effective divisor as in \eqref{cond},
    it follows that being non-special is equivalent to $h^1(\Fer,\mathcal{O}_{\Fer}(D))=0$.

    In \cite[Conjecture 2.6]{Laf2002}, Laface  proposed a conjecture (see Conjecture \ref{conj1}) 
    which provides a characterization for special divisors using the notion of 
    $(-1)$-special divisors (see Algorithm \ref{algo}). 
    We propose the following conjecture (see the next section for the motivation for this 
    conjecture and similarities with conjectures on $\mathbb P^2$): 
    \begin{conjecture}\label{conj3}
		The following holds on $\Fer$:
		\begin{enumerate}[label=(\alph*)]
			\item\label{conj3a} Effective nef divisors on $\Fer$ are non-special.
			\item\label{conj3b} Every integral curve $C$ on $\Fer$ with $C^2<0$ is either a $(-1)$-curve or the strict transform of $C_e$ on $\Fer$.
		\end{enumerate}
	\end{conjecture}

    We state a few consequences of the above conjecture. Under the hypothesis that 
    Conjecture \ref{conj3}\ref{conj3b} is true for all $r>0$, 
    \cite[Theorem 3.6]{HJNS24} proves the existence of irrational 
    Seshadri constants. Moreover, Conjecture \ref{conj3}\ref{conj3b} clearly
    implies the bounded negativity conjecture  for $\Fer$ (which states that for a smooth
    projective surface $X$, there is a constant $b(X)$, such that for all integral
    curves $C\subset X$, we have $C^2>b(X)$). 

    It is well known that the analogue of the above 
    conjecture for $\mb P^2$ (the SHGH conjecture) 
    is true when $r \leqslant 9$ \cite{Nag61}. Note that $\mb P^2$
    blown up at $r$ general points is anticanonical ($h^0(X,-K_X)>0$)
    if and only if $r\leqslant 9$. 
    The main goal of this article is to obtain 
    similar results for Conjecture \ref{conj3}. 
    Conjecture \ref{conj3}\ref{conj3b} was first investigated by 
    \cite{HJSA}, who proved it for $r\leqslant e+2$. 
    Later this was improved to $r\leqslant e+4$ in \cite[Theorem 5.5]{JKS25}. 
    In this article we show that 
    Conjecture \ref{conj3} is true when $r\leqslant e+5$. 
    In Proposition \ref{eq-conjectures} we prove that Conjecture \ref{conj3}
    is equivalent to \cite[Conjecture 2.6]{Laf2002}.
    The following is the main result of this article.

    \begin{theorem}[Theorem \ref{e+4}]\label{thm e+4 intro}
        Let $e\geqslant 0$ and let $r\leqslant e+5$. Conjecture \ref{conj3} holds for $\Fer$.
    \end{theorem}

    \begin{remark}\label{rem e=0}
    Let $e\in\{0,1\}$. 
    The blow-up of $\mathbb F_e$
    at $r$ very general points is isomorphic to the blow-up 
    of $\mathbb P^2$ at $r+1$ very general points. 
    Note that if $r\leqslant e+5$, then $r+1\leqslant e+6\leqslant 7$.
    Thus, when $e\in\{0,1\}$, and $r\leqslant e+5$, then
    $\mathbb F_{e,r}$ is isomorphic to $\mathbb P^2$ blown up 
    at $r+1$ very general points where $r+1\leqslant 7$. 
    Thus, Theorem \ref{thm e+4 intro} is a 
    consequence of the fact that the
    SHGH conjecture for $\mathbb P^2$ is 
    true when $r\leqslant 9$. In view of 
    this, throughout we assume that $e>0$. 
    \end{remark}

    \begin{remark}
    Let $e\in\{0,1\}$. Assume Conjecture \ref{conj3} 
    holds for $\Fer$. Let $X_{r+1}$ denote the blowup
    of $\mathbb P^2$ at $r+1$ very general points. 
    Since $\Fer$ is isomorphic to $X_{r+1}$, from \ref{conj3}(a) it follows
    that \ref{shgh-harbourne}(a) holds for $X_{r+1}$. 
    Since $e\in \{0,1\}$, and the point we are blowing up
    on $\Fer$ are very general, it follows that the strict 
    transform of $C_e$ is a $(-1)$-curve. Thus, \ref{conj3}(b)
    implies \ref{shgh-harbourne}(b) holds for $X_{r+1}$.
    Thus, Conjecture \ref{conj3} for $\Fer$ implies SHGH
    conjecture for $X_{r+1}$. Similarly, it is easily checked
    that SHGH conjecture for $X_{r+1}$ implies Conjecture \ref{conj3}
    for $\Fer$.
    \end{remark}
\noindent 
{\bf Acknowledgements}. Discussions with Antonio Laface 
helped in simplifying some proofs
and also improving the result from $r\leqslant e+4$ to 
$r\leqslant e+5$. We are 
very grateful to him for his interest and his suggestions. 
We are very grateful to the referees for their suggestions. 

\section{Motivation and Preliminaries}\label{section2}
Consider the blow-up of $\mathbb P^2$ at $r$ very general points, denote it by $X_r$. 
We define the \emph{virtual dimension}, \emph{expected dimension}, 
and \emph{non-specialness} of a divisor on $X_r$ analogously to 
Definitions~\ref{defn} and~\ref{non-spcl defn}.
Recall that an integral curve $C$ on a smooth projective surface is a $(-1)$-curve if 
it is a smooth rational curve with $C^2=-1$. From the genus formula it follows 
that $C$ is a $(-1)$-curve if and only if $C^2=-1$ and $K\cdot C=-1$, 
where $K$ is the canonical divisor of the smooth surface.
Let $D$ be an effective 
divisor on $X_r$. We say that $D$ is \emph{$(-1)$-special} if there 
exists a $(-1)$-curve $E$ on $X_r$ such that
$D \cdot E < -1$.
The following statements are equivalent
and are referred to as the SHGH conjecture for $X_r$:
\begin{enumerate}
    \item (Segre~\cite{Seg62}) Let $D$ be an effective divisor on $X_r$. If the general 
    curve of the linear system $|D|$ on $X_r$ is reduced, then $D$ is non-special.
    \item (Hirschowitz~\cite{Hir89}) Let $D$ be an effective divisor on $X_r$. Then $D$ is 
    special if and only if $D$ is $(-1)$-special.
    \item (Harbourne~\cite{Har92}) \label{shgh-harbourne}The following statements hold on $X_r$:
		\begin{enumerate}
			\item Every effective nef divisor on $X_r$ is non-special.
			\item Every integral curve $C \subset X_r$ with $C^2 < 0$ is a $(-1)$-curve.
		\end{enumerate}
\end{enumerate}
There are more equivalent versions of the above conjectures. For the statements of 
these and for the equivalences, we refer the reader to \cite{Har86}, \cite{Gim87} 
and \cite{CM01}. In an interesting recent article, 
\cite[Theorem 2]{Laf24}, the authors prove that on the blow-up of $\mathbb P^2$
at $r$ very general points, (3a) implies (3b).

Motivated by the above conjectures on blow-ups of $\mathbb P^2$, one may ask if it is
possible to formulate and prove similar conjectures for blow-ups of Hirzebruch surfaces. 
Hanumanthu et al. \cite{HJSA} proposed the following conjecture motivated by 
Segre's version of the SHGH conjecture for $\mb P^2$. Let $\widetilde{C}_e$ be the strict transform of $C_e$.
Note that when $e>0$, since the points are in general position, it follows that 
$C_e=\widetilde{C}_e$. 
    
	\begin{conjecture}\label{conj2}\cite[Conjecture 4.4]{HJSA}
		Let $D$ be a divisor as in \eqref{cond} which is effective.
		If a general curve in $|D|$ on $\mathbb{F}_{e,r}$ is reduced 
		and $\widetilde{C}_e$ is not a fixed component of $|D|$, then $D$ is non-special.
	\end{conjecture}

Before stating the analogue of the second version (Hirschowitz's version) of the SHGH conjecture, we recall 
the definition of a $(-1)$-special divisor on $\Fer$ given by 
	Laface in \cite{Laf2002}. Consider the following algorithm:
	
	\begin{algorithm}\label{algo}
		Let $D$ be as in \eqref{cond} and effective.  
		\begin{enumerate}
			\item \label{algo1} If $E$ is a $(-1)$-curve with $-t =D\cdot E<0$, replace $D$ by $D - tE$. 
			\item \label{algo2} If $D \cdot  \widetilde{C}_e < 0$, replace $D$ by $D -\widetilde{C}_e$. 
			\item \label{algo3} If $ D \cdot  \widetilde{C}_e \geqslant 0$ and $D\cdot E\geqslant 0$ for every $(-1)$-curve $E$, then stop. Else, go to Step \ref{algo1}.
		\end{enumerate}
	\end{algorithm}
	
		\begin{remark}\label{rem1}
		Note that since $D$ is an effective divisor, if $D\cdot C<0$ for 
		some integral curve $C$, then $C$ is a fixed component of $|D|$. 
		Hence, this process will end after finitely many steps, as in each 
		step we remove a divisor from the base locus of $|D|$. 
		Let $M$ denote the divisor obtained after the above procedure is 
		complete. It is clear that $\text{dim}|M|=\text{dim}|D|$.
	\end{remark}

	\begin{definition}
		Let $D$ be a divisor as in \eqref{cond}. 
		We say $D$ is \textit{$(-1)$-special} if $v(M) > v(D)$, 
		where $M$ is the divisor that we obtain after applying Algorithm \ref{algo}.
	\end{definition}

The following conjecture, put forward by Laface, is analogous to Hirschowitz's version
of the SHGH conjecture:
	\begin{conjecture}\cite[Conjecture 2.6]{Laf2002}\label{conj1}
		Let $D$ be a divisor as in \eqref{cond} and assume it is effective. 
		Then $D$  is special if and only if it is $(-1)$-special.
	\end{conjecture}
    
	Dumnicki has proved Conjecture \ref{conj1} when $m_i=m\leqslant 8$ \cite{Dum10}.

Finally, we mention that Conjecture \ref{conj3} is analogous to the version of the SHGH 
conjecture given by Harbourne \ref{shgh-harbourne}.

	We recall, for $e\in \mb Z_{\geqslant 0}$, $\mb F_e=\mb P(\mc O_{\mb P^1}\oplus \mc O_{\mb P^1}(-e))$
	denotes the Hirzebruch surface with invariant $e$. 
	Recall from the introduction the divisors $C_e$ and $f$. 
    The canonical divisor class on $\Fe$ 
	is $$K_{\Fe}=-2C_e-(e+2)f\,.$$ The canonical divisor class on $\Fer$ is 
	$$K_{\Fer}=-2C_e-(e+2)f+E_1+E_2+\dots+E_r\,.$$

	\begin{lemma}\label{-1 spcl is spcl}
		Let $D$ be a divisor as in \eqref{cond}, which is effective. If $D$ is $(-1)$-special, then $D$ is special.
	\end{lemma}
	\begin{proof}
		Suppose that $D$ is $(-1)$-special. Let $M$ denote the divisor that we obtain after applying Algorithm \ref{algo}.
		Then by Remark \ref{rem1}, we have 
		\begin{equation}\label{-1spcl eqn}
			\text{dim}(|D|) = \text{dim} (|M|) \geqslant v({M})>v(D).       
		\end{equation} 
		
		Since $D$ is effective, we have $\dim |D|>-1$. Recall that 
		$e(D)=\max\left\{ v(D),-1\right\}$. If $e(D)=-1$, then $D$ 
		is a special divisor as $\dim |D|> e(D)$. If $e(D)=v(D)$ 
		then $D$ is special by \eqref{-1spcl eqn}.
	\end{proof}


	In the following proposition, we list the implications among the 
	three conjectures stated above (Conjecture \ref{conj3}, Conjecture \ref{conj1}
	and Conjecture \ref{conj2}).
	\begin{proposition}\label{eq-conjectures}
		We have the following relationships among the above listed conjectures.
		\begin{enumerate}
			\item Conjecture \ref{conj1} implies Conjecture \ref{conj2}.
			\item Conjecture \ref{conj2} implies Conjecture \ref{conj3}\ref{conj3b}.
			\item Conjecture \ref{conj1} is equivalent to Conjecture \ref{conj3}.
		\end{enumerate}
	\end{proposition}
	\begin{proof}
		For (1), see \cite[Theorem 4.6]{HJSA}. For (2), see \cite[Theorem 4.9]{HJSA}.
		For (3), notice that if $D$ is effective and nef, then after applying 
		Algorithm \ref{algo}, we get $M=D$, and so $D$ is not $(-1)$-special. 
		By Conjecture 
		\ref{conj1} it follows that it is non-special, that is, 
		Conjecture \ref{conj3}\ref{conj3a} holds. Combined with the first two assertions 
		it follows that Conjecture \ref{conj1} implies Conjecture \ref{conj3}.
		
		The following elementary result is very well known. Let $D$ be an 
		effective divisor on $\Fer$. If $D\cdot C\geqslant 0$ for all 
		integral curves $C$ with $C^2<0$, then $D$ is nef.
		
		We will prove that Conjecture \ref{conj3} implies Conjecture \ref{conj1}.
		By Lemma \ref{-1 spcl is spcl}, every effective $(-1)$-special divisor is special.
		So it suffices to show that every effective special divisor $D$ is $(-1)$-special.
		
		
		Let $M$ denote the divisor that we obtain after applying 
		Algorithm \ref{algo} to $D$. So we have $M\cdot E\geqslant 0$ 
		for all $(-1)$-curves $E$ and $M\cdot \widetilde{C}_e\geqslant 0$. 
		Note that by Conjecture \ref{conj3}\ref{conj3b}, 
		if $C$ is an integral curve with $C^2<0$, then $C$ is either a 
		$(-1)$-curve or it is $\widetilde{C}_e$. 
		Hence, $M$ is nef by the elementary result stated above.
		It is clear from Remark \ref{rem1} that $M$ is effective. 
		By Conjecture \ref{conj3}\ref{conj3a}, $M$ is non-special. 
		So $\dim|M|=e(M)=\max\{v(M),-1\}=v(M)$. Since $D$ is special, we have
		$$v(M)=\dim|M|=\dim|D|>e(D)\geqslant v(D).$$
		So $D$ is a $(-1)$-special divisor.
	\end{proof}

    \begin{remark}
        In analogy with the case of $\mathbb P^2$, we expect that Conjecture \ref{conj3}, 
        Conjecture \ref{conj2} and Conjecture \ref{conj1} are equivalent. The preceding 
        Proposition proves the equivalence of Conjecture \ref{conj3} and 
        Conjecture \ref{conj1}. We do not see how to prove that Conjecture \ref{conj2} 
        and Conjecture \ref{conj1} are equivalent.
    \end{remark}

        \begin{lemma} We have the following:
        \begin{enumerate}
            \item $h^0(\Fe, \mc O_{\Fe}(-K_{\Fe}))= 9$ for $e=0,1$,
            \item $h^0(\Fe, \mc O_{\Fe}(-K_{\Fe}))= e+6$ for $e\geqslant 2$.
        \end{enumerate}
        \end{lemma}
        \begin{proof}
            From \cite[Lemma 2.7]{JKS25}, we have 
            $$\pi_{\ast}(2C_e)=\mc {O}_{\mb {P}^1}\oplus\mc {O}_{\mb {P}^1}(-e)
                \oplus\mc {O}_{\mb {P}^1}(-2e)\,.$$
            Hence, using the projection formula we have
            \begin{equation*}
                \begin{split}
                    h^0(\Fe,O_{\Fe}(-K_{\Fe}))&=h^0(\Fe,O_{\Fe}(2C_e+(e+2)f)\\
                    &=h^0(\mb P^1,\pi_{\ast}(\mc O_{\Fer}(2C_e))\otimes \mc O_{\mb P^1}(e+2))\\
                    &=h^0(\mb P^1,\mc O_{\mb P^1}(e+2)\oplus \mc O_{\mb P^1}(2) \oplus \mc O_{\mb P^1}(-e+2))\\
                    \end{split}
            \end{equation*}
            Now it is easy to see that $h^0(\Fe, \mc O_{\Fe}(-K_{\Fe}))= 9$ for $e=0,1$ and $h^0(\Fe, \mc O_{\Fe}(-K_{\Fe}))= e+6$ for $e\geqslant 2$.
        \end{proof}
The above lemma shows that $\Fer$ is an anticanonical surface if $r\leqslant e+5$. 

The following Proposition will be very useful for us. It may be well known to experts, however, 
we include a proof as we could not find a reference. We begin by describing the setup.
Let $1\leqslant r\leqslant e+5$ be an integer. 
Consider the embedding of $\Fe$ into $\mathbb P(H^0(\mc O_{\Fe}(C_e+(e+2)f)))$. 
If we choose the $r$ points to be very general, then there is a general hyperplane
passing through these $r$ points, which meets $\Fe$ in a smooth curve $C$ containing these $r$
points. Let $C'$ denote the strict transform of $C$ in $\Fer$. The class of $C'$ is 
$\widetilde{C}_e+(e+2)f-\sum_{i=1}^rE_i$. Then we have 
$-K_{\Fer}=\widetilde{C}_e+C'$. Also note that $\widetilde{C}_e\cdot C'=2$ and since the 
hyperplane defining $C$ is general, $\widetilde{C}_e$ and $C'$ meet transversely. 
Let $Y_r:=\widetilde{C}_e\cup C'$ be the union of $\widetilde{C}_e$ and $C'$ inside $\Fer$.
After identifying $\widetilde{C}_e$ and $C'$ with $\mb P^1$, we may think of 
$Y_r$ as the union of two copies of $\mb P^1$, where, for 
$i\in \{0,1\}$, the point $i$ on $\widetilde{C}_e$ is glued with the 
point $i$ on $C'$. Given a line bundle $L$ on $Y_r$ such that $L\vert_{\widetilde{C}_e}$ 
and $L\vert_{C'}$ are trivial, we may assume that the trivializations are 
such that the  isomorphism between the fibers $L\vert_{\widetilde{C}_e,0}$ and 
$L\vert_{C',0}$ is the identity map. Then the isomorphism between 
$L\vert_{\widetilde{C}_e,1}$ and $L\vert_{C',1}$ is multiplication by $\theta\in \mb C^*$,
which determines $L$ completely. 
Using this it is easy to see that there is the following short exact sequence 
$$0\to \mb C^*\to {\rm Pic}(Y_r)\to {\rm Pic}(\widetilde{C}_e)\oplus {\rm Pic}(C')\to 0\,.$$

\begin{proposition}\label{injectivity}
    Let $1\leqslant r\leqslant e+5$ be an integer. With the setup as described 
    above, the restriction map of Picard groups 
    ${\rm Pic}(\Fer)\to {\rm Pic}(Y_r)$ is an inclusion. 
\end{proposition}
\begin{proof}
    It is easy to check that the map ${\rm Pic}(\Fe)\to {\rm Pic}(C_e)\oplus {\rm Pic}(C)$
    is an inclusion. 
    
    Let us consider the case $r=1$. 
    Let $p$ be a point on the curve $C$ and consider the blow-up at $p$, which we 
    denote $\mb F_{e,1}$. Let $E$ denote the exceptional divisor. A line 
    bundle on $\mb F_{e,1}$ is given by 
    $\mc O_{\mb F_{e,1}}(a\widetilde{C}_e+bf+cE)$. 
    If the restrictions $\mc O_{\mb F_{e,1}}(a\widetilde{C}_e+bf+cE)\vert_{\widetilde{C}_e}$ 
    and $\mc O_{\mb F_{e,1}}(a\widetilde{C}_e+bf+cE)\vert_{C'}$ 
    are trivial, then we get $b=ae$ and $c=-a(e+2)$.
    Let $L$ be a line bundle whose restriction to $Y_1$ is
    trivial. Then further restrictions to $\widetilde{C}_e$ and $C'$ are trivial. 
    It follows that $L=\mc O_{\mb F_{e,1}}(a(\widetilde{C}_e+eF_1)-a(e+2)E)$. 

    The line bundle $\mc O_{\mb F_{e,1}}(\widetilde{C}_e+ef)$ is globally generated
    and its intersection with $\widetilde{C}_{e}$ is 0. 
    Let $s$ be a general global section of $\mc O_{\mb F_{e,1}}(\widetilde{C}_e+ef)$.
    It follows that the restriction of $s$ to $\widetilde{C}_e$ gives a nowhere vanishing
    global section of $\mc O_{\mb F_{e,1}}(\widetilde{C}_e+ef)\vert_{\widetilde{C}_e}$.
    Similarly, if $s'$ is a global section of $\mc O_{\mb F_{e,1}}((e+2)E)$, then 
    as the intersection of $\widetilde{C}_e$ and $E$ is zero, it follows that 
    the restriction of $s'$ to $\widetilde{C}_e$ gives a nowhere vanishing
    global section of $\mc O_{\mb F_{e,1}}((e+2)E)\vert_{\widetilde{C}_e}$.
    After multiplying with a scalar, we may assume that 
    $(s/s')\vert_{\widetilde{C}_e}$ is the constant function 1 on $\widetilde{C}_e$.
    Next we analyze $(s/s')$ restricted to $C'$. 
    
    Using our identification of $C'$ with $\mb P^1$, we may assume that $s$ vanishes
    at points $\{\lambda_1,\ldots,\lambda_{e+2}\}\in \mb P^1\setminus \{0,1,\infty\}$.
    Let $\{p'\}= E\cap C'$.
    We may assume that $p'$ corresponds to a point 
    $\lambda \in \mb P^1\setminus \{0,1,\infty\}$. 
    It follows that the meromorphic function $(s/s')^a\vert_{C'}$
    equals a scalar multiple of 
    $$\frac{\prod_{i=1}^{e+2}(z-\lambda_i)^a}{(z-\lambda)^{a(e+2)}}\,.$$
    The scalar is determined by the condition that this function evaluated 
    at 0 equals 1. Thus, 
    $$f(z):=\left(\frac{s}{s'}\right)^a
    \Big\vert_{C'}=\mu^{-1}\frac{\prod_{i=1}^{e+2}(z-\lambda_i)^a}{(z-\lambda)^{a(e+2)}}\,,
        \qquad \qquad \text{where }\mu=\frac{\prod_{i=1}^{e+2}(-\lambda_i)^a}{(-\lambda)^{a(e+2)}}\,.$$
    The condition that $L$ restricted to $Y_1$ is trivial now becomes
    \begin{equation}\label{equality of values}
    f(0)=1=f(1)\,.    
    \end{equation}
    Clearly, this function does not satisfy \eqref{equality of values} if $\lambda$ is very general. 
    This proves that ${\rm Pic}(\mb F_{e,1})\to {\rm Pic}(Y_1)$ is an inclusion. 
    
    Let us assume that we have proved the Proposition for all $r$ such that  $1\leqslant r<e+5$. We will
    now prove it for $r+1$. We will slightly abuse notation and use same notations to denote the strict transforms
    of $C_e$ and $C$ in $\Fer$ and $\mb F_{e,r+1}$. We may write 
    ${\rm Pic}(\mb F_{e,r+1})={\rm Pic}(\Fer)\oplus \mb Z[E_r-E_{r+1}]$.
    By our hypothesis, the map ${\rm Pic}(\Fer)\to {\rm Pic}(Y_r)$ is an inclusion. 
    Let $K_r$ denote the kernel of the map ${\rm Pic}(\Fer)\to {\rm Pic}(\widetilde{C}_e)\oplus {\rm Pic}(C')$. 
    We have a commutative diagram 
    \begin{equation}\label{cd kernel}
    \xymatrix{
    0\ar[r] & K_r\ar[r]\ar[d] & {\rm Pic}(\Fer) \ar[r]\ar[d] & {\rm Pic}(\widetilde{C}_e)\oplus {\rm Pic}(C')\ar@{=}[d]\\
    0\ar[r] & \mb C^*\ar[r] & {\rm Pic}(Y_r) \ar[r] & {\rm Pic}(\widetilde{C}_e)\oplus {\rm Pic}(C')
    }
    \end{equation}
    Then $K_r$ is a subgroup of $\mb C^*$. 
    We have a commutative diagram 
    \[\xymatrix{
    {\rm Pic}(\Fer) \ar[r]\ar[d] & {\rm Pic}(Y_r) \ar[d]^\cong\ar[r] & {\rm Pic}(\widetilde{C}_e)\oplus {\rm Pic}(C')\ar[d]^\cong\\
    {\rm Pic}(\mb F_{e,r+1})\ar[r] & {\rm Pic}(Y_{r+1})\ar[r] & {\rm Pic}(\widetilde{C}_e)\oplus {\rm Pic}(C')
    }
    \]
    in which all the maps are pullbacks of line bundles. 
    
    Consider the subgroup $\tilde K_r\subset  \mb C^*$
    consisting of $\alpha\in \mb C^*$ such that there is some 
    positive integer $n$ for which $\alpha^n\in K_r$. 
    This is a countable subgroup of $\mb C^*$. 
    Since $\mc O_{\mb F_{e,r+1}}(E_r-E_{r+1})$ restricted to 
    $\widetilde{C}_e$ and $C'$ are trivial, it follows
    that if the restriction of $\mc O_{\mb F_{e,r+1}}(E_r-E_{r+1})$ to $Y_{r+1}$ is outside 
    $\tilde K_r$, then the map ${\rm Pic}(\mb F_{e,r+1})\to {\rm Pic}(Y_{r+1})$ is an inclusion.
    We argue as in the preceding paragraph. Let $s$ be a nonzero global section of $\mc O_{\mb F_{e,r+1}}(E_r)$
    and let $s'$ be a nonzero global section of $\mc O_{\mb F_{e,r+1}}(E_{r+1})$.
    Then we get a meromorphic function on $Y_{r+1}$,
    $$f(z):=\mu^{-1}\frac{(z-\lambda_r)}{(z-\lambda_{r+1})}\,,\qquad \mu=\frac{(-\lambda_r)}{(-\lambda_{r+1})}\,.$$
    The image of $\mc O_{\mb F_{e,r+1}}(E_r-E_{r+1})$ in $\mb C^*$ is 
    $$f(1)=\frac{(-\lambda_{r+1})}{(-\lambda_{r})}\frac{(1-\lambda_r)}{(1-\lambda_{r+1})}=\frac{(1-\frac{1}{\lambda_r})}{(1-\frac{1}{\lambda_{r+1}})}\,.$$
    We may choose $\lambda_{r+1}$ to be very general, 
    so that this number is not in $\tilde K_r$. 
    Since $K_{r+1}$ is the subgroup generated by $K_r$
    and $f(1)$, it follows that when we consider 
    the diagram \eqref{cd kernel} for $r+1$, the left
    vertical arrow is an inclusion. Thus, it follows that 
    ${\rm Pic}(\mathbb F_{e,r+1})\to {\rm Pic}(Y_{r+1})$ is an inclusion. 
    This completes the proof of the Proposition. 
\end{proof}
    
    \section{Case $r\leqslant e+5$}
        In this section, we prove the Conjecture \ref{conj3} when  $r\leqslant e+5$. 
        Throughout this section, unless stated otherwise, we assume $e>0$ and $r\leqslant e+5$. 

        \begin{theorem}\label{e+4}
            Let $r\leqslant e+5$. Then 
            Conjecture \ref{conj3} holds on $\Fer$. 
        \end{theorem}
        \begin{proof}
        Since $r\leqslant e+5$, we can find a smooth curve on $\Fe$
        passing through $p_1,p_2,\dots,p_r$. 
        Let $C'$ denote the strict transform of this smooth curve on $\Fer$. Note that 
                \begin{itemize}
                    \item $-K_{\Fer}=\widetilde{C}_e+C'$,
                    \item $\widetilde{C}_e\cdot C'=2$,
                    \item $(-K_{\Fer})\cdot C'\geqslant 1$.
                \end{itemize} 
        
        We will use \cite[Theorem I.1]{Har97}. 
        Let $D$ be a nef and effective divisor on $\Fer$. 
        If $D\cdot (-K_{\Fer})>0$ then $h^1(\Fer, \mc O_{\Fer}(D))=0$. So let us assume that 
        $D\cdot (-K_{\Fer})=0$ and $h^1(\Fer, \mc O_{\Fer}(D))>0$. Applying \cite[Theorem I.1]{Har97} 
        it follows that a general section of $D$ has a connected component
        which is disjoint from $\widetilde{C}_e\cup C'$. Let $C$ be an integral
        curve in this connected component. Then $C$ is disjoint from 
        $\widetilde{C}_e$ and $C'$. It follows that the restriction of $\mc O_{\Fer}(C)$
        to $\widetilde{C}_e\cup C'$ is trivial. This contradicts 
        Proposition \ref{injectivity}. This proves Conjecture \ref{conj3}\ref{conj3a}.

        Let $C$ be an integral curve with $C^2<0$ and $C\neq \widetilde C_e$. Note that 
        $$(-K_{\Fer})\cdot C\geqslant 0\,.$$ 
        If $(-K_{\Fer})\cdot C=0$, then we get $C\neq C'$, and so $\widetilde C_e\cdot C=C'\cdot C=0$. 
        It follows that the restriction of $\mc O_{\Fer}(C)$ to $-K_{\Fer}$ is trivial,
        which contradicts Proposition \ref{injectivity}. Thus, $(-K_{\Fer})\cdot C\geqslant 1$.
        $$0>C^2\geqslant -2+(-K_{\Fer})\cdot C\geqslant -1\,.$$
        It follows that $C$ is a $(-1)$-curve. This proves Conjecture \ref{conj3}\ref{conj3b}.
        \end{proof}
    \bibliographystyle{halpha}
	\bibliography{Hir.bib}
\end{document}